\documentclass[11pt,english]{article}
\usepackage[utf8]{inputenc}
\usepackage[T1]{fontenc}
\usepackage{babel}
 \usepackage{amsmath}
 \usepackage{mathtools}
\usepackage{amsfonts}
\usepackage{amssymb}
\usepackage{graphicx}

\usepackage{amsthm}
\usepackage{bbm}
\usepackage{mathtools}
\usepackage[colorlinks=true,
            linkcolor=blue,
            urlcolor=blue,
            citecolor=blue]{hyperref}
\usepackage{authblk}
\usepackage{geometry}
\usepackage{algorithm}
\usepackage[noend]{algpseudocode}
\usepackage[vcentermath]{youngtab}
\usepackage{pgfplots}
\pgfplotsset{compat=newest}
\usetikzlibrary{calc}
\usepackage{mathrsfs}
\definecolor {processblue}{cmyk}{0.96,0,0,0}

\DeclarePairedDelimiter\floor{\lfloor}{\rfloor}
 \author{Mohamed Slim Kammoun\footnote{mohamed-slim.kammoun@univ-lille.fr
 \\ Univ. Lille, CNRS, UMR 8524 - Laboratoire Paul Painlevé, F-59000 Lille, France.}}
\providecommand{\keywords}[1]{\textbf{\textit{Keywords:}} #1}
\title{On the longest common subsequence of independent random permutations invariant under conjugation}
\newtheorem{theorem}{Theorem}
\newtheorem{corollary}[theorem]{Corollary}
\newtheorem{lemma}[theorem]{Lemma}
\newtheorem{proposition}[theorem]{Proposition}

\theoremstyle{definition}

\usepackage{tikz}
\usetikzlibrary{automata, positioning}
\usepackage{natbib}
\newcommand{\e}{\mathbb{E}}
\newcommand{\p}{\mathbb{P}}
\newcommand{\s}{\mathfrak{S}_n}
\newcommand{\inin}{\{1,\dots,n\}}

\newcommand{\bigslant}[2]{{\raisebox{.2em}{$#1$}\left/\raisebox{-.2em}{$#2$}\right.}}

\begin{document}
\maketitle
\begin{abstract} 
 \cite{MR3509473} conjectured that the expectation of the length of the  longest common subsequence of two i.i.d random permutations of size $n$ is greater than $\sqrt{n}$. We prove in this paper that there exists a universal constant $n_1$ such that their  conjecture is satisfied for  any pair of i.i.d random permutations of size greater than $n_1$ with distribution invariant under conjugation. 
 We prove also that asymptotically, this expectation is at least of order $2\sqrt{n}$ which is the asymptotic behaviour of the uniform setting. More generally, in the case where the laws of the two permutations are not necessarily the same, we gibe a lower  bound for the expectation. 
In particular, we prove that if one of the permutations is invariant under conjugation and with a good control of the expectation of the number of its cycles, the limiting fluctuations of the length of the longest common subsequence  are of  Tracy-Widom type. This result holds independently of the law of the second permutation.
\end{abstract}
\section{Introduction and main results}
\paragraph{}
Let $\s$ be the symmetric group, namely the  group of permutations of $\inin$. Given $\sigma \in \s$,  $(\sigma(i_1),\dots,\sigma(i_k))$ is a subsequence of $\sigma$ of length $k$ if $ i_1<i_2<\dots<i_k$.  We denote by $LCS(\sigma_1,\sigma_2)$ the length of the longest common subsequence (LCS) of two permutations. % Common subsequences for words are used for example in computer science in file comparison and screen redisplay. 
\paragraph{}
In the sequel  of this article, we consider two sequences of random permutations ${(\sigma_{1,n})}_{n\geq 1}$ and ${(\sigma_{2,n})}_{n\geq 1}$ with joint distribution $\p$ and associated expectation $\e$  such that $\sigma_{1,n}$ and $\sigma_{2,n}$ are independent and supported on  $\s$.
The study of the $LCS$ of independent random permutations was initiated  by \cite{houdre2014central}. Recently \cite{MR3830132} showed that for i.i.d random permutations 
\begin{align*}
    \e(LCS(\sigma_{1,n},\sigma_{2,n}))\geq \sqrt[3]{n}.
\end{align*}
It is conjectured by \citep{MR3509473} that for i.i.d  random permutations,
\begin{align*}
    \e(LCS(\sigma_{1,n},\sigma_{2,n}))\geq \sqrt{n}.
\end{align*}
 In this article, we obtain asymptotic  bounds  in the scale of  $\sqrt{n}$ in the case where the law of at least one of the two permutations is invariant under conjugation.  We say that the law of $\sigma_n$ is\textit{ invariant under conjugation }if for any $\hat \sigma \in \s$, $\hat \sigma \circ\sigma_n \circ {\hat \sigma}^{-1}$  is equal in distribution to $\sigma_n$. 
\subsection{LCS of two independent random permutations with  distribution invariant under conjugation}
\paragraph*{}
In Theorem~\ref{thm6}, we give an asymptotic lower bound for the $LCS$ of two independent random permutations. Under a good control of the number of fixed points, we give a better  bound  in Proposition~\ref{thm7}. Finally, as an  application of Proposition~\ref{thm7}, we give an  asymptotically optimal lower bound for   i.i.d random permutations with distributions invariant under conjugation in  Corollary~\ref{thmp}.
\begin{theorem}\label{thm6}
Assume that for any $n\geq 1$, $\sigma_{1,n}$ and $\sigma_{2,n}$  are independent  and their distributions   are invariant under conjugation. Then
\begin{align*}
    \liminf_{n\to\infty}\frac{\e(LCS(\sigma_{1,n},\sigma_{2,n}))}{\sqrt{n}}\geq 2\sqrt{\theta} \simeq 0.564,
\end{align*}
where  $\theta$ is the unique solution of $G(2\sqrt{x})=\frac{2+x}{12}$,
\begin{align}
\nonumber
    G:= [0,2]&\to\left[0,\frac{1}{2}\right]
    \\x &\mapsto
    \int_{-1}^{1}\left(\Omega(s)- \left|s+\frac{x}{2}\right| - \frac{x}{2}\right)_+\mathrm{d} s,
    \label{e2}
\end{align}
and
\begin{align*}
\Omega(s):=\begin{cases}
\frac{2}{\pi}(s\arcsin({s})+\sqrt{1-s^2}) & \text{ if } |s|<1 \\ 
|s| & \text{ if } |s|\geq 1 
\end{cases}.
\end{align*}
\end{theorem}
The function $\Omega$ appears as the Vershik-Kerov-Lagan-Shepp limit shape. For more details, one can  see Figure~\ref{vkls} and Lemma~\ref{lem10}. We will prove this result in Subsection~\ref{pthm6} by comparison with the uniform distribution on $\s$ and the uniform distribution on the set of involutions. 
\paragraph{}
Under a good control of the number of fixed points, we  obtain a better bound. 
\begin{proposition}\label{thm7}
Let $0\leq \alpha\leq 2$.  Assume that for any $n\geq 1$, $\sigma_{1,n}$ and $\sigma_{2,n}$  are independent  and their distributions   are invariant under conjugation.
\begin{itemize}
    \item[-]  If
\begin{equation}\label{cond1}
    \lim_{n\to\infty} \max(\p(\sigma_{1,n}(1)=1),\p(\sigma_{2,n}(1)=1))=0,
\end{equation}
then
\begin{align} \label{res}
\liminf_{n\to\infty}\frac{\e(LCS(\sigma_{1,n},\sigma_{2,n}))}{\sqrt{n}}\geq 2.
\end{align}
\item[-]
If
\begin{equation} \label{cond2}
    \liminf_{n\to\infty} \sqrt{n}{\p(\sigma_{1,n}(1)=1)\p(\sigma_{2,n}(1)=1)}\geq \alpha,
\end{equation}
 then
\begin{align} 
\liminf_{n\to\infty}\frac{\e(LCS(\sigma_{1,n},\sigma_{2,n}))}{\sqrt{n}}\geq \alpha.
\end{align}
\end{itemize}

\end{proposition}
Consequently, we obtain  the following result for i.i.d random permutations.
\begin{corollary}
\label{thmp}
Assume that  for any $n\geq 1$, $\sigma_{1,n}$ and $\sigma_{2,n}$  are two independent and identically distributed  random permutations with distribution invariant under conjugation. Then
\begin{align*}
    \liminf_{n\to \infty}\frac{  \e(LCS(\sigma_{1,n},\sigma_{2,n}))}{\sqrt{n}}\geq2.
\end{align*}
\end{corollary}
We conjecture that we can get rid of  \eqref{cond1} and \eqref{cond2}; the stability under conjugation is sufficient to obtain \eqref{res} which is equivalent to  replace $2\sqrt\theta$ by $2$ in Theorem~\ref{thm6}.  
We will prove   Proposition~\ref{thm7}  and Corollary~\ref{thmp} in Subsection~\ref{pthm7}. The idea of  the proof is  to study the longest increasing subsequence of $\sigma_{1,n}^{-1}\circ \sigma_{2,n}$ knowing that under a good control of the number of fixed points of the two permutations, the number of cycles of $\sigma_{1,n}^{-1}\circ \sigma_{2,n}$ is sufficiently small to compare it with the uniform distribution.  
\subsection{LCS of two independent random permutations where one of the   distributions is invariant under conjugation }
\paragraph{}
When  $\sigma_{2,n}$ is not invariant under conjugation, we give an asymptotic lower  bound of $ \frac{\e(LCS(\sigma_{1,n},\sigma_{2,n}))}{\sqrt{n}}$ in Theorem~\ref{thm8}. Moreover, we prove in Proposition~\ref{thm5} that under a good  control of the number of cycles of $\sigma_{1,n}$,   $\lim_{n\to \infty} \frac{\e(LCS(\sigma_{1,n},\sigma_{2,n}))}{\sqrt{n}}=2$ and under a stronger control, we have Tracy-Widom fluctuations for $LCS(\sigma_{1,n},\sigma_{2,n})$.
\begin{theorem} \label{thm8}
Assume that for any $n\geq 1$, $\sigma_{1,n}$ and $\sigma_{2,n}$  are independent  and  the law  of $\sigma_{1,n}$ is invariant under conjugation. Then 
\begin{align*}
    \liminf_{n\to \infty} \frac{\e(LCS(\sigma_{1,n},\sigma_{2,n}))}{\sqrt{n}}\geq G^{-1} \left({\liminf_{n\to\infty}\frac{\e(\#(\sigma_{1,n}))}{2n}}\right),
\end{align*}
where $\#(\sigma)$ is the number of cycles of $\sigma$ and $G$ is defined in \eqref{e2}.\\
In particular, if  
$\lim_{n\to\infty}\e\left(\frac{\#(\sigma_{1,n})}{n}\right)=0$, we have 
\begin{align*} 
\liminf_{n \to \infty} \frac{\e\left({LCS(\sigma_{1,n},\sigma_{2,n})}\right)}{\sqrt{n}}\geq2.
\end{align*}
\end{theorem}
\begin{proposition} \label{thm5}
Assume that for any $n\geq 1$, $\sigma_{1,n}$ and $\sigma_{2,n}$  are independent  and  the law  of $\sigma_{1,n}$ is invariant under conjugation.\begin{itemize}
    \item[-]
 If 
$\frac{\#(\sigma_{1,n})}{\sqrt[6]{n}}\overset{\mathbb{P}}\to 0,   
$
then $\forall s\in\mathbb{R}$,
$$\lim_{n \to \infty} \mathbb{P}\left(\frac{LCS(\sigma_{1,n},\sigma_{2,n})-2\sqrt{n}}{\sqrt[6]{n}}\leq s\right)=F_2(s),
$$ where  $F_2$ is  the cumulative distribution function of the  Tracy-Widom distribution. 
\item[-] If 
$\frac{\#(\sigma_{1,n})}{\sqrt{n}}\overset{\mathbb{P}}\to 0,   
$ then
$    \frac{LCS(\sigma_{1,n},\sigma_{2,n})}{\sqrt{n}} \overset{\mathbb{P}}\to 2.
$
\item [-] If  
$\lim_{n\to\infty}\e\left(\frac{\#(\sigma_{1,n})}{\sqrt{n}}\right)=0,   
$ then
$\lim_{n \to \infty} \frac{\e\left({LCS(\sigma_{1,n},\sigma_{2,n})}\right)}{\sqrt{n}}=2. \quad 
$\end{itemize}

\end{proposition}
Note that in Theorem \ref{thm8} and in Proposition~\ref{thm5},  we do not have any assumption on the distribution of $\sigma_{2,n}$. The proof in Subsection~\ref{pthm8} is based on a coupling argument    between  $\sigma_{1,n}$ and  a uniform  permutation.
\section*{Acknowledgements}
\paragraph*{}
The author would like to acknowledge many extremely useful conversations
with  Mylène Maïda, Adrien Hardy and Christan Houdré and their great help to improve the coherence of this paper. He would also  acknowledge a useful discussion with Pierre-Loïc Méliot about Gelfand measures. This work is partially supported  by the Labex CEMPI (ANR-11-LABX-0007-01).
\section{Proof of results}
\subsection{General tools}
\paragraph{}
Given $\sigma \in \s$ and
$1\leq i_1<i_2<\dots<i_k\leq n$, the subsequence $(\sigma(i_1),\dots,\sigma(i_k))$  is an increasing subsequence of $\sigma$ if  $\sigma(i_1)<\dots<\sigma(i_k)$. We denote by $\ell(\sigma)$   the length of the longest increasing subsequence of $\sigma$. \\ For example, for the permutation \begin{equation*}\sigma=\begin{pmatrix}
 1& 2 & 3 & 4 & 5 \\ 
 5& 3 & 2 & 1 & 4 
\end{pmatrix},\end{equation*} we have $\ell(\sigma)=2$.
The study of the longest common subsequence is strongly related to the notion of  longest  increasing subsequence. More precisely, we have the following.   
\begin{proposition} \label{prop}
Let $\sigma_1,\sigma_2\in\s$. 
\begin{align*}
    LCS(\sigma_1,\sigma_2)=\ell(\sigma_1^{-1}\circ\sigma_2)=\ell(\sigma_2^{-1}\circ\sigma_1).
\end{align*}
\end{proposition}

\begin{proof}
It is clear that the length of the longest common subsequence is invariant under left composition. Consequently 
\begin{align*}
    LCS(\sigma_1,\sigma_2)=LCS(\sigma_1^{-1}\circ\sigma_1,\sigma_1^{-1}\circ\sigma_2)=LCS(Id_n,\sigma_1^{-1}\circ\sigma_2).
\end{align*}

Observe that by definition, the subsequences of $Id_n$ are the increasing subsequences which concludes the proof. 
\end{proof}
 We will use in the remainder of this paper the Robinson–Schensted correspondence \citep*{RSKR,MR0121305} or the Robinson–Schensted–Knuth correspondence \citep*{RSKK}.    We denote  by $$\lambda(\sigma)=\{\lambda_i(\sigma)\}_{i\geq 1}.$$ the shape of the image of $\sigma$ by  this correspondence. 
  We will not include here detailed description of the algorithm. For further reading, we recommend \citep*[Chapter 3]{Sagan2001}. 
\paragraph{}
We denote by  \begin{align*}
\mathfrak{I}_1(\sigma):&=\{s\subset\{1,2,\dots,n\};\; \forall i,j \in s,\; (i-j)(\sigma(i)-\sigma(j))\geq 0 \},
\\\mathfrak{I}_{k+1}(\sigma):&=\{s\cup s',\; s\in \mathfrak{I}_k,\;s'\in \mathfrak{I}_1\}.
\end{align*}
The link to the longest increasing subsequence is given by the following result.\begin{lemma} \label{RSKLEMMA} \citep*{GREENE1974254}
For any permutation $ \sigma\in \mathfrak{S}_n$,
\begin{align*}
\max_{s\in \mathfrak{I}_i(\sigma)} |s| =\sum_{k=1}^i \lambda_k(\sigma).
\end{align*}
In particular, $$\ell(\sigma)=\max_{s\in \mathfrak{I}_1(\sigma)} |s| =\lambda_1(\sigma).$$
\end{lemma}
 Let $L_{\lambda(\sigma)}$ be the height function of $\lambda(\sigma)$   rotated by $\frac{3\pi}{4}$  and extended by the function $x\mapsto |x|$ to obtain a function defined on $\mathbb{R}$.\\ For example,  if $\lambda(\sigma)=(7,5,2,1,1,\underline{0})$, then the associated function $L_{\lambda(\sigma)}$ is represented by Figure \ref{figL}. 
\begin{figure}[H]
\centering
\begin{tikzpicture}    [/pgfplots/y=0.5cm, /pgfplots/x=0.5cm]
      \begin{axis}[
    axis x line=center,
    axis y line=center,
    xmin=0, xmax=10,
    ymin=0, ymax=10, clip=false,
    ytick={0},
	xtick={0},
    minor xtick={0,1,2,3,3,4,5,6,7,8,9},
    minor ytick={0,1,2,3,3,4,5,6,7,8,9},
    grid=both,
    legend pos=north west,
    ymajorgrids=false,
    xmajorgrids=false, anchor=origin,
    grid style=dashed    , rotate around={45:(rel axis cs:0,0)}
,
]

\addplot[
    color=blue,
        line width=3pt,
    ]
    coordinates {
    (0,10)(0,7)(1,7)(1,5)(2,5)(2,2)(3,2)(3,1)(5,1)(5,0)(10,0)
    };
 
\end{axis}
\begin{axis}[
	axis x line=center,
    axis y line=center,
    xmin=-7.07, xmax=7.07,
    ymin=0, ymax=8, anchor=origin, clip=false,
    xtick={-7,-6,-5,-4,-3,-2,-1,0,1,2,3,4,5,6,7},
    ytick={0,1,2,3,3,4,5,6,7,8},
    legend pos=north west,
    ymajorgrids=false,
    xmajorgrids=false,rotate around={0:(rel axis cs:0,0)},
    grid style=dashed];
\end{axis}
    \end{tikzpicture}
    \caption{ $L_{(7,5,2,1,1,\underline{0})}$}
     \label{figL}
\end{figure}
The image of the uniform permutation by the Robinson-Schensted correspondence is known as the Plancherel measure. Its typical shape was studied separately by \cite{LOGAN1977206} and  \cite{MR0480398}. Stronger results have been proved by \cite{Vershik1985}.  In 1993, Kerov studied the limiting fluctuations  but did not publish his results. One can see \citep{10.1007/978-94-010-0524-1_3} for further  details.

\paragraph{}
To prove our results,  we will use the Markov operator  $T$ defined on $\s$ and associated to the stochastic matrix 
$\left[\frac{\mathrm{1}_{A_{\sigma_1}}(\sigma_2)}{card(A_{\sigma_1})}\right]_{\sigma_1,\sigma_2 \in \s}$
where 
\begin{equation*}
    A_\sigma =  \begin{cases} 
    \{\sigma\} & \text{if } \#(\sigma)=1\\ 
      \{\rho \in \s, \sigma^{-1}\circ\rho=(i_1,i_2)\circ(i_1,i_3)\dots\circ(i_1,i_{\#(\sigma)}) \text{ and } \#(\rho)=1 \} &\text{if } \#(\sigma)>1 
    \end{cases}.
\end{equation*}
We recall that $\#(\sigma)$ is the number of cycles of $\sigma$.
$T$ is then the Markov operator mapping a permutation $\sigma$ to a permutation uniformly chosen  among the permutations obtained by merging the cycles of $\sigma$ using transpositions having all a common point.   
Note that $A_\sigma$ is not empty since any choice of one point in each cycle  gives a possible $(i_1,i_2,\dots i_{\#(\sigma)})$ and a correspondent  permutation $\rho$.
\begin{lemma}\label{basiclem}
For any permutation $\sigma$,
\begin{itemize}
    \item[-] Almost surely, 
    \begin{equation}\label{eq13}
    |\ell(T(\sigma))-\ell(\sigma)|\leq \#(\sigma).
        \end{equation}
    \item[-] More generally, almost surely,
\begin{align}\label{eq14} 
\max_{i\geq 1} \left|\sum_{k=1}^i\left( \lambda_k(\sigma)-\lambda_k\left(T(\sigma)\right)\right)\right|\leq \#(\sigma).
\end{align}
\end{itemize}
Moreover, for any  random permutation $\sigma_n$ invariant under conjugation on $\s$, the law of $T(\sigma_n)$ is the uniform distribution on permutations with a unique cycle.
\end{lemma}
Note that the  uniform distribution on permutations with a unique cycle is also known as the Ewens's distribution with parameter $0$. We denote it by $Ew(0)$.
\begin{proof}
The law of $T(\sigma_n)$ is clearly invariant under conjugation. Indeed, let $\sigma,\rho\in\s$.
\begin{align*}
    \p(T(\sigma_n)=\sigma)&=\mathbf{1}_{\#(\sigma)=1} \sum_{\hat{\sigma}\in\s}\mathbf{1}_{\sigma \in A_{\hat\sigma}} \frac{\p(\sigma_n=\hat\sigma)}{card(A_{\hat\sigma})}
    \\ &=\mathbf{1}_{\#(\sigma)=1} \sum_{\hat{\sigma}\in\s}\mathbf{1}_{{\rho\circ\sigma\circ\rho^{-1}} \in A_{{\rho\circ\hat\sigma\circ\rho^{-1}}}} \frac{\p({\rho\circ\sigma_n\circ\rho^{-1}}={\rho\circ\hat\sigma\circ\rho^{-1}})}{card(A_{\rho\circ\hat\sigma\circ\rho^{-1}})}
    \\ &=\mathbf{1}_{\#(\sigma)=1} \sum_{\hat{\sigma}\in\s}\mathbf{1}_{{\rho\circ\sigma\circ\rho^{-1}} \in A_{{\hat\sigma}}} \frac{\p({\rho\circ\sigma_n\circ\rho^{-1}}=\hat\sigma)}{card(A_{\hat\sigma})}
    \\&=\mathbf{1}_{\#({\rho\circ\sigma\circ\rho^{-1}})=1} \sum_{\hat{\sigma}\in\s}\mathbf{1}_{{\rho\circ\sigma\circ\rho^{-1}} \in A_{{\hat\sigma}}} \frac{\p({\sigma_n}=\hat\sigma)}{card(A_{\hat\sigma})}
    \\&=   \p(T(\sigma_n)=\rho\circ\sigma\circ\rho^{-1}).
    \end{align*}
    Moreover, by construction, almost surely,
$\#(T(\sigma_n))=1$. Consequently, the law of  $T(\sigma_n)$ is $Ew(0)$.
\paragraph{}
Let $\sigma$ be  a permutation.  By definition of $\ell(\sigma)$,  there exists ${i_1<i_2<\dots<i_ {\ell(\sigma)}}$ such that ${\sigma(i_1)<\dots<\sigma(i_{\ell(\sigma)})}$. Let $\rho=\sigma\circ(j_1,j_2)\circ(j_1,j_3)\dots\circ(j_1,j_{\#(\sigma)})$ be a permutation with a unique cycle and $i'_1,i'_2,\dots,i'_m$ be   the same sequence as $i_1,i_2,\dots, i_{\ell(\sigma)}$ after removing $j_1$, $j_2$, \dots, $j_{\#(\sigma)}$  if needed. We have   $\ell(\sigma)-\#(\sigma)\leq m$ and $\sigma(i'_1)<\dots<\sigma(i'_{m})$. Knowing that  $\forall i\notin \{j_1,j_2,\dots ,j_{\#(\sigma)}\}$, $\rho (i)=\sigma(i)$, so that  $$\rho (i'_1)<\dots<\rho (i'_{m}).$$ Therefore,  $m\leq \ell(\rho)$ and 
\begin{align*}
\ell(\sigma)-\ell(\rho )\leq \#({\sigma}).
\end{align*}
We can obtain the reverse inequality in \eqref{eq13} using the same techniques.
Similarly, to prove \eqref{eq14}, let $l\geq1$ and $\left\{i_1,i_2,\dots, i_{\sum_{k=1}^l \lambda_k(\sigma)}\right\}\in\mathfrak{I}_l(\sigma).$  Let
$i'_1,i'_2,\dots,i'_m$ be   the same sequence as $i_1,i_2,\dots, i_{\ell(\sigma)}$ after removing $j_1$, $j_2$, \dots  $,j_{\#(\sigma)}$  
if needed. We have   $\left\{i'_1,i'_2,\dots, i'_{m}\right\}\in\mathfrak{I}_l(\rho) $ and we conclude as in the proof of \eqref{eq13}.  
\end{proof}
For more details, one can see \citep{kammoun2018}. We used the same techniques of proof with a different Markov operator. Here, the bound is better thanks to the use of  the same point $i_1$ to merge cycles. 
\begin{lemma}\citep[Theorem~1.8]{kammoun2018}\label{lem10}
Assume that the distribution of $\sigma_n$ is $Ew(0)$.
Then for all $\varepsilon>0$,
\begin{align*}
\lim_{n\to \infty} \mathbb{P}\left(\sup_{s\in \mathbb{R}} \left|\frac{1}{\sqrt{2n}}L_{\lambda(\sigma_n)}\left({s}{\sqrt{2n}}\right)-\Omega(s)\right|<\varepsilon\right) =1,
\end{align*}
where we recall that \begin{align*}
\Omega(s):=\begin{cases}
\frac{2}{\pi}(s\arcsin({s})+\sqrt{1-s^2}) & \text{ if } |s|<1 \\ 
|s| & \text{ if } |s|\geq 1 
\end{cases}.
\end{align*}
\end{lemma}
 For the remainder of this paper, we will refer to this limiting shape as  the Vershik-Kerov-Logan-Shepp shape. See Figure \ref{vkls}\footnote{This figure is generated by DPPy \citep*{GaBaVa18}}.
This convergence is closely related to the Wigner's semi-circular law. For further details, one can see \citep{ss2,ss1,ss3}.

\begin{figure}[H] 
    \centering
\includegraphics[scale=0.4]{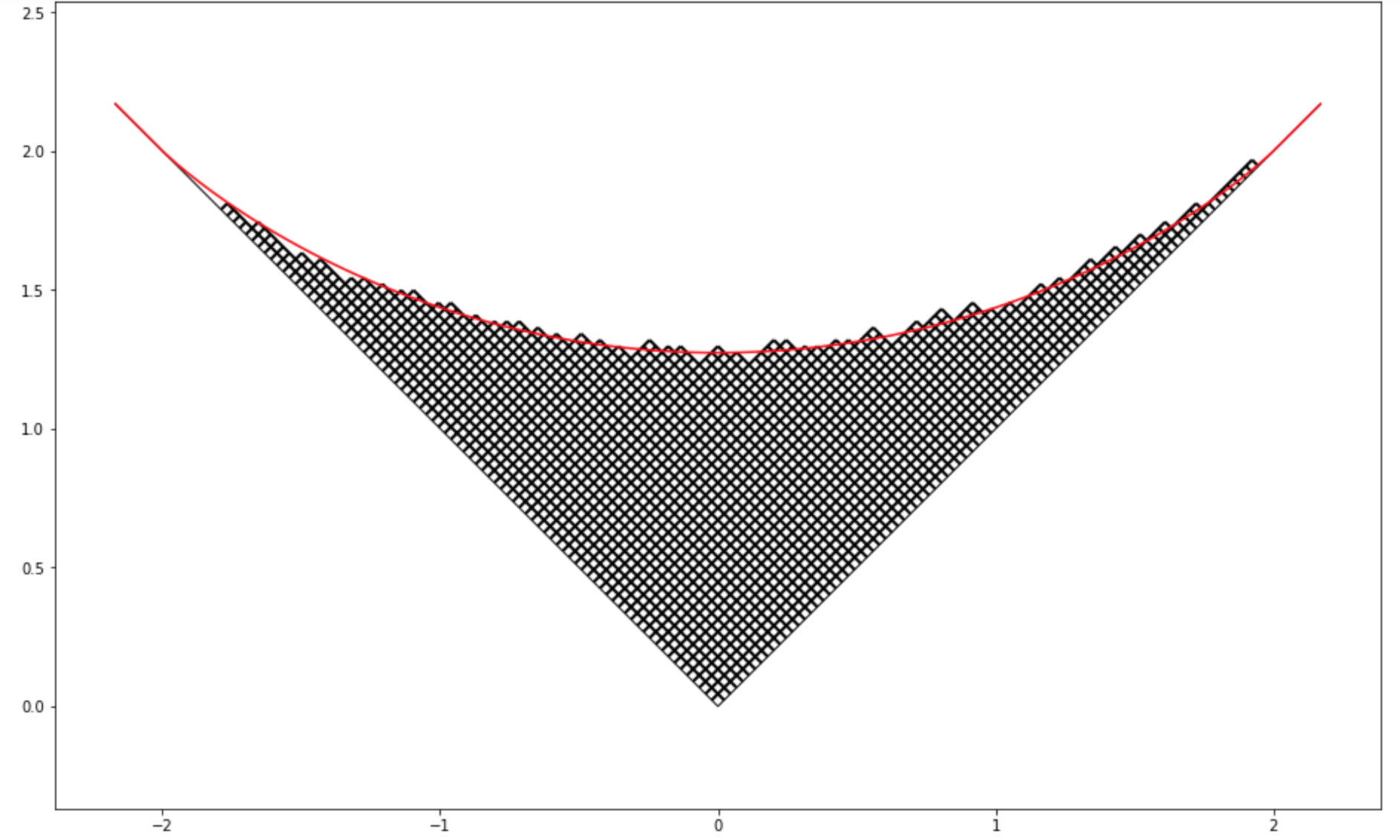}
    \caption{Illustration of the Vershik-Kerov-Logan-Shepp convergence }
    \label{vkls}
\end{figure}
\begin{corollary}\label{col11}
Assume that the distribution of $\sigma_n$ is $Ew(0)$. Then for any $0\leq\gamma\leq2$, for any $\varepsilon>0$,
\begin{equation}\label{gamma}
\p\left(\frac{\sum_{i=1}^{n}(\lambda_i(\sigma_n)-\gamma\sqrt{n})_{+}}{n}> 2 G(\gamma)-\varepsilon \right) \to 1.
\end{equation}
\end{corollary}
\begin{proof}
This is a direct application of Lemma \ref{lem10}. One can see that 
$\frac{\sum_{i=1}^{n}(\lambda_i(\sigma)-\gamma\sqrt{n})_{+}}{2n}$ is the area of the region delimited by  the curves of the functions $x\mapsto |x|$, $x \mapsto\gamma+x$ and $x\mapsto\frac{ L_{\lambda(\sigma)}(x\sqrt{2n})}{\sqrt{2n}}$, see Figure \ref{figL2figure}. By construction, this area is equal to
 $$ \int_{-\infty}^{\infty}\left(\frac{ L_{\lambda(\sigma)}(s\sqrt{2n})}{\sqrt{2n}}- \left|s+\frac{\gamma}{2}\right| - \frac{\gamma}{2}\right)_+\mathrm{d} s.$$
By Lemma \ref{lem10},  
\begin{align*}
    \int_{-1}^{1}\left(\frac{ L_{\lambda(\sigma)}(s\sqrt{2n})}{\sqrt{2n}}- \left|s+\frac{\gamma}{2}\right| - \frac{\gamma}{2}\right)_+\mathrm{d} s\overset{\p}{\to}G(\gamma).
\end{align*} We can  conclude then that
\begin{align*}
    \frac{\sum_{i=1}^{n}(\lambda_i(\sigma)-\gamma\sqrt{n})_{+}}{n}&=2 \int_{-\infty}^{\infty}\left(\frac{ L_{\lambda(\sigma)}(s\sqrt{2n})}{\sqrt{2n}}- \left|s+\frac{\gamma}{2}\right| - \frac{\gamma}{2}\right)_+\mathrm{d} s 
\\& \geq 2 \int_{-1}^{1}\left(\frac{ L_{\lambda(\sigma)}(s\sqrt{2n})}{\sqrt{2n}}- \left|s+\frac{\gamma}{2}\right| - \frac{\gamma}{2}\right)_+\mathrm{d} s 
\\ & \overset{\mathbb{P}}{\to}2 G(\gamma).
\end{align*}
This yields  \eqref{gamma}.
\end{proof}

\begin{figure}[H]
\centering
\begin{tikzpicture}    [/pgfplots/y=0.6cm, /pgfplots/x=0.6cm]
      \begin{axis}[
    axis x line=center,
    axis y line=center,
    xmin=0, xmax=10,
    ymin=0, ymax=10, clip=false,
    ytick={0},
	xtick={0},
    minor xtick={0,1,2,3,3,4,5,6,7,8,9},
    minor ytick={0,1,2,3,3,4,5,6,7,8,9},
    grid=both,
    legend pos=north west,
    ymajorgrids=false,
    xmajorgrids=false, anchor=origin,
    grid style=dashed    , rotate around={45:(rel axis cs:0,0)}
,
]

\addplot[
    color=blue,
        line width=3pt,
    ]
    coordinates {
    (0,10)(0,7)(1,7)(1,5)(2,5)(2,2)(3,2)(3,1)(5,1)(5,0)(10,0)
    };
    \draw [fill=cyan,cyan] (0,4) rectangle (2,5);
    \draw [fill=cyan,cyan] (0,4) rectangle (1,7);

\end{axis}
\begin{axis}[
	axis x line=center,
    axis y line=center,
    xmin=-7.07, xmax=7.07,
    ymin=0, ymax=8, anchor=origin, clip=false,
    xtick={0},
    ytick={0},
     minor ytick={0,1.41,2.83,4.24,5.65},
    legend pos=north west,
    ymajorgrids=false,
    xmajorgrids=false,rotate around={0:(rel axis cs:0,0)},
    grid style=dashed];
    \addplot[
    color=red,
        line width=3pt,
    ]
    coordinates {
    (-5.65,0)(0,5.65)
    };
        \addplot[
    color=red,
        line width=3pt,
    ]
    coordinates {
    (0,0)(-8,8)
    };
    \node (premier) at (0.3,5.65) {1};
    \node (second) at (0.5,2.82) {0.5};
    \node (secondz) at (0.7,1.41) {0.25};
    \node (secondez) at (0.7,4.24) {0.75};
\end{axis}

    \end{tikzpicture}
    \caption{ $\lambda=(7,2,2,1,1,\underline{0})$ and $ \gamma=1$} 
         \label{figL2figure}

\end{figure}
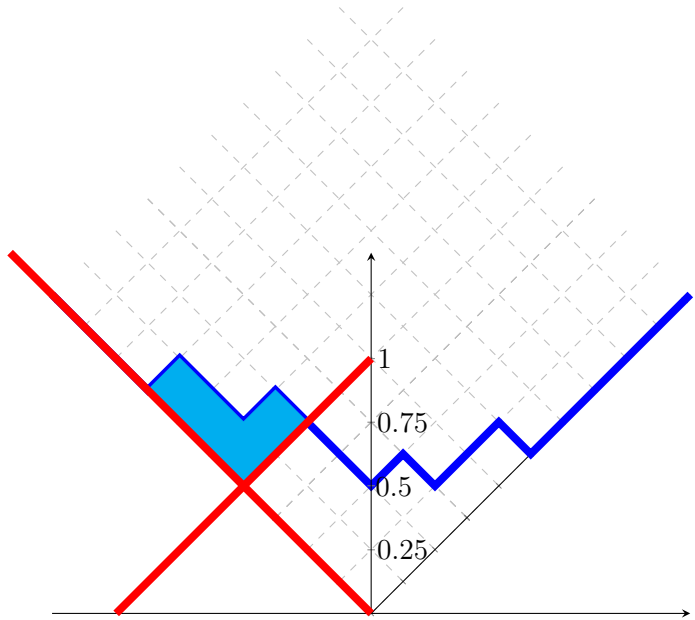
Note that it is not difficult to prove that $$\frac{\sum_{i=1}^{n}(\lambda_i(\sigma_n)-\gamma\sqrt{n})_{+}}{n}\overset{\p}\to 2 G(\gamma).$$
We skip the proof here as we only need \eqref{gamma} in the sequel. 

\begin{corollary} \label{col13}
For any permutation $\sigma$, for any $\alpha \geq 0$,  almost surely,
\begin{align*}
\left|{\sum_{i=1}^{\infty}(\lambda_i(\sigma)-\alpha\sqrt{n})_{+}} - {\sum_{i=1}^{\infty}(\lambda_i(T(\sigma))-\alpha\sqrt{n})_{+}}\right|  \leq {\#(\sigma)}.
\end{align*}
\end{corollary}
\begin{proof}We prove first that
\begin{align*}
{\sum_{i=1}^{\infty}(\lambda_i(\sigma)-\alpha\sqrt{n})_{+}} - {\sum_{i=1}^{\infty}(\lambda_i(T(\sigma))-\alpha\sqrt{n})_{+}}  \leq {\#(\sigma)}.
\end{align*}
If  $\lambda_1(\sigma)\leq\alpha\sqrt{n}$, the inequality is trivial as the right hand side is non-negative and the left hand side is non-positive. Otherwise, let $k:= \max \{j\geq1, \lambda_j(\sigma)> \alpha\sqrt{n}\}.$
We have
\begin{align*}
{\sum_{i=1}^{\infty}(\lambda_i(\sigma)-\alpha\sqrt{n})_{+}} &=
{\sum_{i=1}^k(\lambda_i(\sigma)-\alpha\sqrt{n})_{+}} +{\sum_{i=k+1}^{\infty}(\lambda_i(\sigma_n)-\alpha\sqrt{n})_{+}}
\\&=\sum_{i=1}^k {(\lambda_i(\sigma)-\alpha\sqrt{n})},
\end{align*} 
and\begin{align*}
    {\sum_{i=1}^{\infty}(\lambda_i(T(\sigma))-\alpha\sqrt{n})_{+}} &\geq {\sum_{i=1}^{k}(\lambda_i(T(\sigma))-\alpha\sqrt{n})_{+}} 
    \geq \sum_{i=1}^k{( \lambda_i(T(\sigma))-\alpha\sqrt{n})}.
\end{align*}
Using \eqref{eq14}, we obtain 
\begin{align*}
{\sum_{i=1}^{\infty}(\lambda_i(\sigma)-\alpha\sqrt{n})_{+}} - {\sum_{i=1}^{\infty}(\lambda_i(T(\sigma))-\alpha\sqrt{n})_{+}} &\leq \sum_{i=1}^k{ \lambda_i(\sigma)-\lambda_i(T(\sigma))} \\&\leq {\#(\sigma)}.
\end{align*}
The reverse inequality is obtained by exchanging the role of $\sigma$ and $T(\sigma)$.
\end{proof}
\begin{corollary}\label{col14}
For any $\alpha<2$, there exist $\beta>0$ and $n_\alpha>0$  such that for any $n>n_\alpha$, for any random permutation $\sigma_n$ invariant under conjugation satisfying  
$\mathbb{E}(\#\sigma_n)<n\beta$, we have
\begin{equation*}
    \e(\ell(\sigma_n))\geq \alpha\sqrt{n}.
\end{equation*}
\end{corollary}

\begin{proof}
This is a direct application of Corollary \ref{col11} and Corollary \ref{col13}.
Let   $\alpha<\gamma<2$, $\varepsilon>0$ and  $\beta>0$  such that  $1-\frac{\beta}{G(\gamma)} - \varepsilon > \frac{\alpha}{\gamma}$. By Corollary \ref{col11}, we obtain  the existence of $n_\alpha$ such that for any $n>n_\alpha$,
\begin{equation*}
\p\left(\frac{\sum_{i=1}^{n}(\lambda_i(T(\sigma_n))-\gamma\sqrt{n})_{+}}{n}> G(\gamma) \right) > 1-\varepsilon.
\end{equation*}
Since  $\{\ell(\sigma)>k\}$ is equivalent to  $\{\sum_{i=1}^\infty (\lambda_i(\sigma)-k)_{+} >0\} $ and by Markov inequality, we obtain
\begin{align*}
    \e(\ell(\sigma_n))&\geq  \gamma\sqrt{n}  \p(\ell(\sigma_n)\geq\gamma\sqrt{n})
    \\& \geq  \gamma\sqrt{n} \p\left(\frac{\sum_{i=1}^{n}(\lambda_i(T(\sigma_n))-\gamma\sqrt{n})_{+}}{n}> G(\gamma) , \frac{\#(\sigma_n)}{n}<G(\gamma) \right)
    \\& \geq 
     \gamma\sqrt{n} \left(1-\frac{\beta}{G(\gamma)}-\varepsilon\right) \\&\geq \alpha\sqrt{n}.
\end{align*}
\end{proof}
\begin{lemma} \label{lem21}

Let $\sigma \in \s$ and $\rho \in A_\sigma$, then
\begin{align*}
    \ell(\sigma)\geq \sup\left\{k\in \mathbb{N},\sum_{i=1}^\infty (\lambda_i(\rho)-k)_{+} \geq \#(\sigma)\right\}
\end{align*}
and
\begin{align*}
    \ell(\rho)\geq \sup\left\{k\in \mathbb{N},\sum_{i=1}^\infty (\lambda_i(\sigma)-k)_{+} \geq \#(\sigma)\right\}.
\end{align*}
\end{lemma}
\begin{proof}
By the equivalence between  $\{\ell(\sigma)>k\}$ and $\{\sum_{i=1}^\infty (\lambda_i(\sigma)-k)_{+} >0 \}$,
this a direct application of Corollary~\ref{col13}.

\end{proof}

\subsection{Proof of Proposition~\ref{thm7} and Corollary~\ref{thmp}} \label{pthm7}
\paragraph{}
To prove  Proposition~\ref{thm7}  and Corollary~\ref{thmp}, we distinguish two cases. For the first case, we suppose that the number of fixed points is large enough.  We use the fact that for a given permutation, the length of the longest increasing subsequence is bigger than the number of fixed points. For the second case, when the number of fixed points is controlled,  we prove in Lemma \ref{lem15} that the number of cycles of $(\sigma_{1,n})^{-1}\circ{\sigma_{2,n}}$ is sufficiently controlled  to use Corollary~\ref{col14}. In both cases, we can conclude
by Proposition~\ref{prop}. 
\begin{lemma}\label{lem15}
For any $ k\geq 2$,  there exists $C, C'>0$ such that for any $n\geq 1$, for any  independent random permutations $\sigma_{1,n}$ and $\sigma_{2,n}$  with distributions invariant under conjugation,
\begin{align*}
    \p\left(c_1\left((\sigma_{1,n})^{-1}\circ{\sigma_{2,n}}\right)=k\right) \leq  \frac{C}{n}+ C' (\p(\sigma_{1,n}(1)=1)+\p(\sigma_{2,n}(1)=1)),
\end{align*}
where $c_m(\sigma)$ is the length of the cycle of $\sigma$ containing $m$.
\end{lemma}
To prove this result, we will introduce some new objects. To a couple of permutations, we will associate a couple of graphs. 
\\ We denote by $\mathbb{G}^n_k$ the set of oriented simple graphs  with vertices $\{1,2,\dots,n\}$ and having exactly $k$ edges.
\\ For example, $\mathbb{G}^2_1= \left\{\begin {tikzpicture}[-latex ,auto ,node distance =1 cm and 1.5cm ,on grid ,
semithick ,
state/.style ={ circle ,top color =white , bottom color = processblue!20 ,
draw,processblue , text=blue , minimum width =0.1 cm}]
\node[state] (C) {$1$};
\node[state] (D) [right =of C] {$2$};
\path (C) edge [bend left =25]  (D); 
\end{tikzpicture}, \begin {tikzpicture}[-latex ,auto ,node distance =1 cm and 1.5cm ,on grid ,
semithick ,
state/.style ={ circle ,top color =white , bottom color = processblue!20 ,
draw,processblue , text=blue , minimum width =0.1 cm}]
\node[state] (C) {$2$};
\node[state] (D) [right =of C] {$1$};
\path (C) edge [bend left =25]  (D); 
\end{tikzpicture},  
\begin {tikzpicture}[-latex ,auto ,node distance =1 cm and 1.5cm ,on grid ,
semithick ,
state/.style ={ circle ,top color =white , bottom color = processblue!20 ,
draw,processblue , text=blue , minimum width =0.1 cm}]
\node[state] (C) {$1$};
\node[state] (D) [below =of C] {$2$};
\path  (C) edge [loop above]  (C); 
\end{tikzpicture},
\begin {tikzpicture}[-latex ,auto ,node distance =1 cm and 1.5cm ,on grid ,
semithick ,
state/.style ={ circle ,top color =white , bottom color = processblue!20 ,
draw,processblue , text=blue , minimum width =0.1 cm}]
\node[state] (C) {$2$};
\node[state] (D) [below =of C] {$1$};
\path  (C) edge [loop above]  (C); 
\end{tikzpicture}
\right\}$ .
\\ Given  $g\in\mathbb{G}^n_k$, we denote by $E_g$ the set of its edges and by $A_g:=[\mathbbm{1}_{(i,j)\in E_g}]_{1\leq i,j\leq n}$ its adjacency matrix. A connected component of $g$ is called \textit{trivial} if it does not have any edge and a vertex $i$ of $g$ is called \textit{isolated} if $E_g$ does not contain any edge of the form $(i,j)$ or $(j,i)$.  
We say that  two  oriented simple graphs $g_1$ and $g_2$ are \textit{isomorphic} if  one can obtain  $g_2$ by changing the labels of the vertices of $g_1$. In particular, if  $g_1,g_2\in\mathbb{G}^n_k$  then $g_1,g_2$ are isomorphic if and  only if   there exists a permutation matrix $\sigma$ such that  $A_{g_1}\sigma=\sigma A_{g_2}$. 
Let $g\in\mathbb{G}^n_k $, we denote by  $\tilde{g}$  the graph obtained  from  $g$  after removing isolated vertices.
Let $\mathcal{R}$ be the equivalence relation  such that $g_1\mathcal{R} g_2$ if  $\tilde{g}_1$ and $\tilde{g}_2$ are isomorphic.
 We denote by  $\hat{\mathbb{G}}_k:=\bigslant{\cup_{n\geq1} \mathbb{G}_k^n}{\mathcal{R}}
$   the set of  equivalence classes of $\cup_{n\geq1} \mathbb{G}_k^n$  for the relation $\mathcal{R}$.
\\
For example,  \ $\begin {tikzpicture}[-latex ,auto ,node distance =1 cm and 1.5cm ,on grid ,
semithick ,
state/.style ={ circle ,top color =white , bottom color = processblue!20 ,
draw,processblue , text=blue , minimum width =0.1 cm}]
\node[state] (C) {$2$};
\node[state] (D) [below =of C] {$1$};
\path  (C) edge [loop above]  (C); 
\end{tikzpicture} \ 
\mathcal{R}
\ 
\begin {tikzpicture}[-latex ,auto ,node distance =1 cm and 1.5cm ,on grid ,
semithick ,
state/.style ={ circle ,top color =white , bottom color = processblue!20 ,
draw,processblue , text=blue , minimum width =0.1 cm}]
\node[state] (C) {$1$};
\path  (C) edge [loop above]  (C); 
\end{tikzpicture}$ and $\hat{\mathbb{G}}_1 = \left\{\begin {tikzpicture}[-latex ,auto ,node distance =1 cm and 1.5cm ,on grid ,
semithick ,
state/.style ={ circle ,top color =white , bottom color = processblue!20 ,
draw,processblue , text=blue , minimum width =0.1 cm}]
\node[state] (C) {$\ $};
\node[state] (D) [right =of C] {$\ $};
\path (C) edge [bend left =25]  (D); 
\end{tikzpicture},  
\begin {tikzpicture}[-latex ,auto ,node distance =1 cm and 1.5cm ,on grid ,
semithick ,
state/.style ={ circle ,top color =white , bottom color = processblue!20 ,
draw,processblue , text=blue , minimum width =0.1 cm}]
\node[state] (C) {$\ $};
\path  (C) edge [loop above]  (C); 
\end{tikzpicture} \right\}$.
\\
Let $n$ be a positive integer and $\sigma_1,\sigma_2\in\s$.
Let $k_m:=c_m(\sigma_{1}^{-1}\circ\sigma_{2})$, $(i^m_1=m,i^m_2,\dots,i^m_{k_m})$ be the cycle  of $\sigma_{1}^{-1}\circ\sigma_{2}$ containing $m$  and $j_l^m:=\sigma_{2}(i^m_l)$. 
In particular, $i_1^m,i_2^m,\dots,i_{k_m}^m$ are pairwise distinct and  $j_1^m,j_2^m,\dots,j_{k_m}^m$  are pairwise distinct. 
We denote by   $\mathcal{G}_1^m (\sigma_1,\sigma_2) \in \mathbb{G}_{k_m}^n$  the  graph such that $E_{\mathcal{G}_1^m (\sigma_1,\sigma_2)}=\{(i^m_1,j^m_{k_m})\} \bigcup \left(\bigcup_{l=1}^{k_m-1}{\{(i^m_{l+1},j^m_l)\}}\right)   $. We denote also by 
$\mathcal{G}_2^m(\sigma_1,\sigma_2)\in \mathbb{G}_{k_m}^n$ the  graph such that $E_{\mathcal{G}_2^m (\sigma_1,\sigma_2)}=\cup_{l=1}^{k_m}{\{(i^m_l,j^m_{l})\}}$. 
\emph{In particular, $\mathcal{G}_1^m(\sigma_1,\sigma_2)$ and $\mathcal{G}_2^m(\sigma_1,\sigma_2)$ have the same set of non-isolated vertices}.
For $i\in\{1,2\}$, let 
$\hat{\mathcal{G}}^m_i(\sigma_1,\sigma_2)$ be the  equivalence class of $\mathcal{G}^m_i(\sigma_1,\sigma_2)$. \\
For example, if $$\sigma_{1}=\begin{pmatrix}
 1& 2 & 3 & 4 & 5 \\ 
 5& 3 & 2 & 1 & 4 
\end{pmatrix} \quad \text{and} \quad \sigma_{2}=\begin{pmatrix}
 1& 2 & 3 & 4 & 5 \\ 
 2& 3 & 5 & 1 & 4 
\end{pmatrix},$$ we obtain $E_{\mathcal{G}_1^1 (\sigma_1,\sigma_2)}=\{(1,5),(3,2)\}$, $E_{\mathcal{G}_2^1 (\sigma_1,\sigma_2)}=\{(1,2),(3,5)\}$, 

\begin {tikzpicture}[-latex ,auto ,node distance =1 cm and 1.5cm ,on grid ,
semithick ,
state/.style ={ circle ,top color =white , bottom color = processblue!20 ,
draw,processblue , text=blue , minimum width =0.1 cm}]
\node[state] (C) {$3$};
\node[state] (A) [above right=of C] {$5$};
\node[state] (B) [above =of C] {$1$};
\node[state] (E) [below =of C] {$4$};
\node[state] (D) [right =of C] {$2$};
\path (B) edge [bend left =25]  (A);
\path (C) edge [bend left =25]  (D);
\node[] at (-1.5,0) {$\mathcal{G}^1_1(\sigma_1,\sigma_2)=$};
\end{tikzpicture}
\begin {tikzpicture}[-latex ,auto ,node distance =1 cm and 1.5cm ,on grid ,
semithick ,
state/.style ={ circle ,top color =white , bottom color = processblue!20 ,
draw,processblue , text=blue , minimum width =0.1 cm}]
\node[state] (C) {$3$};
\node[state] (A) [above right=of C] {$2$};
\node[state] (B) [above =of C] {$1$};
\node[state] (E) [below =of C] {$4$};
\node[state] (D) [right =of C] {$5$};
\path (B) edge [bend left =25]  (A);
\path (C) edge [bend left =25]  (D);
\node[] at (-1.7,0) {, \ \ $\mathcal{G}^1_2(\sigma_1,\sigma_2)=$};
\node[] at (+2.2,0) { \ \ and};
\end{tikzpicture}

\begin {tikzpicture}[-latex ,auto ,node distance =1 cm and 1.5cm ,on grid ,
semithick ,
state/.style ={ circle ,top color =white , bottom color = processblue!20 ,
draw,processblue , text=blue , minimum width =0.6 cm}]
\node[state] (C) {};
\node[state] (A) [above right=of C] {};
\node[state] (B) [above =of C] {};
\node[state] (D) [right =of C] {};
\path (B) edge [bend left =25]  (A);
\path (C) edge [bend left =25]  (D);
\node[] at (-2.3,0.5) {$\hat{\mathcal{G}}^1_1(\sigma_1,\sigma_2)=\hat{\mathcal{G}}^1_2(\sigma_1,\sigma_2)=$};\node[] at (2,0.5) {.};
\end{tikzpicture}
\\ Finally, given  $g\in\mathbb{G}^n_k$, we denote by 
$$\mathfrak{S}_{n,g}:=\{\sigma\in \s; \forall (i,j)\in E_g, \sigma(i)=j \}.$$ 
It is not difficult to prove the two following lemmas.
\begin{lemma} \label{lem20}
If  ${m_1}\in \{i^{m_2}_l, 1\leq l\leq k_{m_2}\}$, then $\mathcal{G}^{m_1}_1(\sigma_1,\sigma_2)=\mathcal{G}^{m_2}_1(\sigma_1,\sigma_2)$ and $\mathcal{G}^{m_1}_2(\sigma_1,\sigma_2)=\mathcal{G}^{m_2}_2(\sigma_1,\sigma_2)$.
\end{lemma}
\begin{proof}
If ${m_1}\in \{i^{m_2}_l, 1\leq l\leq k_{m_2}\}$, then there exists $1\leq l \leq k_{m_1} $ such that $(\sigma^{-1}_{1}\circ\sigma_{2})^l(m_1)=m_2$. Consequently, $k_{m_1}=k_{m_2}$,  $
(i^{m_2}_1,i^{m_2}_2,\dots, i^{m_2}_{k_{m_2}})=(i^{m_1}_l,i^{m_1}_{l+1},\dots, i^{m_1}_{k_{m_1}},i^{m_1}_{{1}},\dots,i^{m_1}_{{l-1}})$ and  $
(j^{m_2}_1,j^{m_2}_2,\dots, j^{m_2}_{k_{m_2}})=(j^{m_1}_l,j^{m_1}_{l+1},\dots, j^{m_1}_{k_{m_1}},j^{m_1}_{{1}},\dots,j^{m_1}_{{l-1}})$ and  we can check easily that  $\mathcal{G}^{m_1}_1(\sigma_1,\sigma_2)=\mathcal{G}^{m_2}_1(\sigma_1,\sigma_2)$ and $\mathcal{G}^{m_1}_2(\sigma_1,\sigma_2)=\mathcal{G}^{m_2}_2(\sigma_1,\sigma_2)$.
\end{proof}
\begin{lemma}\label{lemma15}
Let $g_1, g_2 \in \mathbb{G}^n_k$. Assume that there exists  $\rho \in \s$ 
such that $A_{g_2}\rho=\rho A_{g_1}$. If $\rho$ has  a fixed point on  any non-trivial connected component  of $g_1$, then $\mathfrak{S}_{n,g_1}\cap\mathfrak{S}_{n,g_2}=\emptyset $ or $A_{g_1}=A_{g_2}$.
\end{lemma}
\begin{proof}
Let $\rho \in \s$  be a permutation  having a fixed point on  any non-trivial connected component  of $g_1$ such that $A_{g_2}\rho=\rho A_{g_1}$. Assume that $A_{g_1} \neq A_{g_2}$.  There exists necessarily $(i,j)\in E_{g_1}$ 
such that $\rho(i)=i$  and $\rho(j)\neq j$ or $\rho(j)=j$  and $\rho(i)\neq i$ . In the first case, 
$\mathfrak{S}_{n,g_1}\cap\mathfrak{S}_{n,g_2}\subset\{\sigma\in \s; \sigma(i)=j, \sigma(i)=\rho(j) \}=\emptyset$. 
In the second case, 
$\mathfrak{S}_{n,g_1}\cap\mathfrak{S}_{n,g_2}\subset\{\sigma\in \s; \sigma(i)=j, \sigma(\rho(i))=j \}=\emptyset$.
\end{proof}
The following result is immediate. 
\begin{corollary}
For any graph $g\in \mathbb{G}^n_k$ having $p$ non-trivial connected components and $v$ non-isolated vertices, for any random permutation $\sigma_n$ with distribution invariant under conjugation on $\s$,
\begin{align*}
    \p(\sigma_n\in\mathfrak{S}_{n,g}) \leq \frac{1}{{{n-p}\choose{v-p}}(v-p)!}.
\end{align*}
\end{corollary}
\begin{proof}
If there exist $i,j,l$, with $j\neq l $  such that $\{(i,j)\cup(i,l)\} \subset E_g$ or $\{(j,i)\cup(l,i)\} \subset E_g$  then $\mathfrak{S}_{n,g} =\emptyset$. Therefore, if $\mathfrak{S}_{n,g} \neq \emptyset$, then  non-trivial connected components of $g$ having $w$ vertices  are either cycles of length $w$ or isomorphic to $\overline{g}_w$, where $
A_{\overline{g}_w}=[\mathbbm{1}_{j=i+1}]_{1\leq i,j\leq w}.$\\
For example, $\overline{g}_5=\begin {tikzpicture}[-latex ,auto ,node distance =1 cm and 1.5cm ,on grid ,
semithick ,
state/.style ={ circle ,top color =white , bottom color = processblue!20 ,
draw,processblue , text=blue , minimum width =0.1 cm}]
\node[state] (A) {$1$};
\node[state] (B) [right =of A] {$2$};
\node[state] (C) [right =of B] {$3$};
\node[state] (D) [right =of C] {$4$};
\node[state] (E) [right =of D] {$5$};
\path (A) edge [bend left =25]  (B);
\path (B) edge [bend left =25]  (C);
\path (D) edge [bend left =25]  (E);
\path (C) edge [bend left =25]  (D);
\end{tikzpicture}$.
Let $g\in \mathbb{G}^n_k$ such that $\mathfrak{S}_{n,g} \neq \emptyset$.
Fix  $p$ vertices $x_1,x_2,\dots,x_p$  each belonging  to a different  non-trivial connected components of $g$.  Let $\{x_1,x_2,\dots x_p,\dots,x_v\}$ be the set of non-isolated vertices of $g$. 
Let $$F= \{(y_i)_{p+1\leq i \leq n}; y_i \in \{1,2,\dots,n\}\setminus\{x_1,\dots x_p\} \text{ pairwise distinct}\}.$$ 
Given $y=(y_i)_{p+1\leq i \leq n}\in F$, we denote by $g_y \in \mathbb{G}^n_k$ the graph isomorphic to $g$ obtained by fixing the labels of $ x_1,x_2,\dots,x_p$ and by changing the labels of $x_i$ by $y_i$ for $p+1\leq i\leq v$. Since non trivial connected  components of $g$ of length $w$ are either cycles or isomorphic to $\bar{g}_w$, if $y \neq y' \in F$, then $g_y\neq g_{y'}$ and by Lemma~\ref{lemma15}, $\mathfrak{S}_{n,g_{y}}\cap\mathfrak{S}_{n,g_{y'}}=\emptyset$. Since $\sigma_n$ is invariant under conjugation, we have $ \p(\sigma_n\in\mathfrak{S}_{n,g_{y}})= \p(\sigma_n\in\mathfrak{S}_{n,g_{y'}})= \p(\sigma_n\in\mathfrak{S}_{n,g}).$
Therefore,
\begin{align*}
     \p(\sigma_n\in\mathfrak{S}_{n,g})=\frac{\sum_{y\in F}\p(\sigma_n\in\mathfrak{S}_{n,g_{y}})}{card{(F)}}=\frac{\p(\sigma_n\in\cup_{y\in F}\mathfrak{S}_{n,g_{y}})}{card(F)}\leq \frac{1}{card(F)}= \frac{1}{{{n-p}\choose{v-p}}(v-p)!}.
\end{align*}
\end{proof}
We will now prove Lemma \ref{lem15}.
\begin{proof}[Proof of Lemma \ref{lem15}]
Note that $\hat{\mathbb{G}}_k$ is finite. Therefore, it is sufficient  to prove that for any $\hat{g}_1,\hat{g}_2 \in \hat{\mathbb{G}}_k$ having the same number of vertices, there exist two  constants $C_{\hat{g}_1,\hat{g}_2}$ and $C'_{\hat{g}_1,\hat{g}_2}$ such that for any integer $n$, 
\begin{align*}
    \p((\hat{\mathcal{G}}^1_1(\sigma_{1,n},\sigma_{2,n}),\hat{\mathcal{G}}^1_2(\sigma_{1,n},\sigma_{2,n}))=(\hat{g}_1,\hat{g}_2)))\leq \frac{C_{\hat{g}_1,\hat{g}_2}}{n}+ C'_{\hat{g}_1,\hat{g}_2}(\p(\sigma_{1,n}(1)=1)+\p(\sigma_{2,n}(1)=1)).
\end{align*}
Let $\hat{g}_1,\hat{g}_2 \in \hat{\mathbb{G}}_k$ be two unlabeled graphs having respectively $p_1$ and $p_2$ connected component and   $v\leq 2k$  vertices. Let  $B^n_{\hat{g}_1,\hat{g}_2}$ be the set of couples  $({g}_1,{g}_2) \in (\mathbb{G}^n_k)^2$ having the same non-isolated vertices such that $1$ is a  non-isolated vertex of both graphs and, for $i \in \{1,2\}$, the 
equivalence class  of $g_i$ is  $\hat{g}_i$. 
\begin{itemize}
    \item [-]
Suppose that $\hat{g}_1$ and $\hat{g}_2$ do not contain any loop i.e  no edges of type $(i,i)$.
Then $p_1\leq \frac{v}{2}$ and $p_2\leq \frac{v}{2}$. Consequently,

\begin{align*}
        &\p((\hat{\mathcal{G}}^1_1(\sigma_{1,n},\sigma_{2,n}),\hat{\mathcal{G}}^1_2(\sigma_{1,n},\sigma_{2,n}))=(\hat{g}_1,\hat{g}_2))) 
        \\=& \sum_{(g_1,g_2)\in B^n_{\hat{g}_1,\hat{g}_2}} \p((\mathcal{G}^1_1(\sigma_{1,n},\sigma_{2,n}),\mathcal{G}^1_2(\sigma_{1,n},\sigma_{2,n}))=(g_1,g_2))
        \\\leq& \sum_{(g_1,g_2)\in B^n_{\hat{g}_1,\hat{g}_2}}
        \p(\sigma_{1,n}\in \mathfrak{S}_{n,g_1},
        \sigma_{2,n}\in \mathfrak{S}_{n,g_2})
        \\=& \sum_{(g_1,g_2)\in B^n_{\hat{g}_1,\hat{g}_2}}
        \p(\sigma_{1,n}\in \mathfrak{S}_{n,g_1})
        \p(\sigma_{2,n}\in \mathfrak{S}_{n,g_2})
        \\\leq& \sum_{(g_1,g_2)\in B^n_{\hat{g}_1,\hat{g}_2}} \frac{1}{{{n-p_1}\choose{v-p_1}}(v-p_1)!}
        \frac{1}{{{n-p_2}\choose{v-p_2}}(v-p_2)!}
        \\=&
        \frac{card(B^n_{\hat{g}_1,\hat{g}_2})}{{{n-p_1}\choose{v-p_1}}(v-p_1)!{{n-p_2}\choose{v-p_2}}(v-p_2)!}
        \\ \leq& 
        \frac{{{n-1}\choose{v-1}} {v!}^2 }{{{n-p_1}\choose{v-p_1}}(v-p_1)!{{n-p_2}\choose{v-p_2}}(v-p_2)!}
        \\\leq & C_{g_1,g_2}n^{v-1-(v-p_1+v-p_2)}
        = C_{g_1,g_2}n^{p_1+p_2-v-1}
        \leq \frac{C_{g_1,g_2}}{n}.
\end{align*}
\item[-]
Suppose that $\hat{g}_1$ contains a loop. By Lemma \ref{lem20}, 
if $\hat{\mathcal{G}}^m_1(\sigma_{1},\sigma_{2})=\hat{g}_1$, then there exists $j$ a fixed point of $\sigma_1$ such that $k_j=k$ and  $j\in \{i^{m}_l, 1\leq l\leq k\}$.
Thus,  almost surely, 
\begin{align*}
\sum_{i=1}^n \mathbf{1}_{\hat{\mathcal{G}}^i_1(\sigma_{1,n},\sigma_{2,n})=\hat{g}_1}  &\leq k\ card(\{i \in fix(\sigma_{1,n}); k_{i}=k\})\leq k \ card(fix(\sigma_{1,n})),
\end{align*}
where $fix(\sigma)$ is the set of fixed points of $\sigma$. 
Consequently, since $\sigma_{1,n}$ is invariant under conjugation,
\begin{align*}
\p\left(\left(\hat{\mathcal{G}}^1_1(\sigma_{1,n},\sigma_{2,n}),\hat{\mathcal{G}}^1_2(\sigma_{1,n},\sigma_{2,n})\right)=(\hat{g}_1,\hat{g}_2)\right) &\leq \p\left(\hat{\mathcal{G}}^1_1(\sigma_{1,n},\sigma_{2,n})=\hat{g}_1\right)\\&= \frac{\sum_{i=1}^n \p\left({\hat{\mathcal{G}}^i_1(\sigma_{1,n},\sigma_{2,n})=\hat{g}_1}\right)}{n} \\&\leq k\frac{\e(card(fix(\sigma_{1,n})))}{n}\\&=k\p(\sigma_{1,n}(1)=1).
\end{align*}
Similarly, if $\hat{g}_2$ contains a loop, then 
\begin{align*}
\p\left(\left(\hat{\mathcal{G}}^1_1(\sigma_{1,n},\sigma_{2,n}),\hat{\mathcal{G}}^1_2(\sigma_{1,n},\sigma_{2,n})\right)=(\hat{g}_1,\hat{g}_2)\right)\leq k\p(\sigma_{2,n}(1)=1).
\end{align*}
\end{itemize}
\end{proof}
We will now prove Proposition~\ref{thm7}.
\begin{proof}[Proof of Proposition~\ref{thm7}]
Under the condition of Proposition~\ref{thm7},
\begin{itemize}
    \item [-]Assume that \begin{equation*} 
    \liminf_{n\to\infty} \sqrt{n}{\p(\sigma_{1,n}(1)=1)\p(\sigma_{2,n}(1)=1)}\geq \alpha.
\end{equation*}
In this case, 
    \begin{align*}
   \liminf_{n\to \infty}
   \frac{\e(LCS(\sigma_{1,n},\sigma_{2,n}))}{\sqrt{n}}&\geq \liminf_{n\to \infty}\frac{ \e(card(fix(\sigma_{1,n}\circ\sigma_{2,n}^{-1})))}{\sqrt{n}}
   \\&\geq \liminf_{n\to \infty}
\sqrt{n}{\p(\sigma_{1,n}(1)=1)\p(\sigma_{2,n}(1)=1)}\\&\geq \alpha.
\end{align*}
\item[-] Assume that
\begin{equation}\label{H1}
    \lim_{n\to\infty} \max(\p(\sigma_{1,n}(1)=1),\p(\sigma_{2,n}(1)=1))=0.
\end{equation}
\paragraph{}
In this case, 
\begin{align*}
    \p\left( \sigma^{-1}_{1,n}\circ\sigma_{2,n}(1)=1\right)&=\sum_{i=1}^n \mathbb{P}(\sigma_{1,n}(1)=i)\mathbb{P}(\sigma_{2,n}(1)=i)\\&= \p(\sigma_{1,n}(1)=1)\p(\sigma_{2,n}(1)=1)\\&+\frac{(1-\p(\sigma_{1,n}(1)=1))(1-\p(\sigma_{2,n}(1)=1))}{n-1}
    \\&=o(1).
\end{align*}
For any random permutation $\sigma_n\in \s$ invariant under conjugation,
\begin{align*}
    \e(\#(\sigma_n))=\e \left(\sum_{i=1}^n \frac{1}{c_i(\sigma_n)}\right)=\sum_{i=1}^n \e \left(\frac{1}{c_i(\sigma_n)}\right)=n\e\left(\frac{1}{c_1(\sigma_n)}\right),
\end{align*}
and 
 for $n_\beta:=\floor{\frac{1}{\beta}}+1$, with the same  $\beta$ as in Corollary \ref{col14},
\begin{align*}
    \frac{\e(\#(\sigma_n))}{n}&= \sum_{k=1}^\infty \frac{1}{k}\p(c_1(\sigma_n)=k) \\&\leq 
     \p(c_1(\sigma_n)=1) +\sum_{k=2}^{n_\beta} \p(c_1(\sigma_n)=k) + \frac{1}{n_\beta+1}  \sum_{k=n_\beta+1}^{\infty} \p(c_1(\sigma_n)=k)
     \\ &\leq   \p(\sigma_n(1)=1)+    \sum_{k=2}^{n_\beta} \p(c_1(\sigma_n)=k) + \frac{1}{n_\beta+1}.
\end{align*}
Consequently,  under \eqref{H1}, by Lemma \ref{lem15}, we have
\begin{align*}
    \frac{\e(\#(\sigma_{1,n}\circ\sigma_{2,n}^{-1}))}{n}  & \leq \frac{1}{n_\beta+1}+o(1)< \beta+o(1).
\end{align*}
Hence, we obtain Proposition~\ref{thm7} thanks to Corollary \ref{col14}.
\end{itemize}
\end{proof}
\begin{proof}[Proof of Corollary \ref{thmp}]
This is a direct application of  Proposition~\ref{thm7}. In fact, if 
$$\p(\sigma_{1,n}(1)=1) \geq \frac{\sqrt{2}}{\sqrt[4]{n}},$$
then 
\begin{equation*} 
    \liminf_{n\to\infty} \sqrt{n}{\p(\sigma_{1,n}(1)=1)\p(\sigma_{2,n}(1)=1)}\geq 2.
\end{equation*}
Otherwise, 
\begin{equation*}
    \lim_{n\to\infty} \max(\p(\sigma_{1,n}(1)=1),\p(\sigma_{2,n}(1)=1))=0.
\end{equation*}
\end{proof}
\subsection{Proof of Theorem~\ref{thm6}}  \label{pthm6}
\paragraph{}
By observing  that  if $\sigma_{1,n}$ and $\sigma_{2,n}$ are independent random permutations with distribution invariant under conjugation then $\sigma_{1,n}^{-1}\circ\sigma_{2,n}$ is invariant under conjugation,
proving Theorem~\ref{thm6} is equivalent to prove  the following.
\begin{theorem}\label{thm20}
For any sequence of random permutations $\{\sigma_n\}_{n\geq1}$ invariant under conjugation,
\begin{align*}
    \liminf_{n\to\infty}\frac{\e(\ell(\sigma_n))}{\sqrt{n}}\geq 2\sqrt{\theta}.
\end{align*}
\end{theorem}
The argument will be by comparison with the uniform measure on $\s$ and the uniform measure on the set of involutions. We will use the  uniform permutation on $\s$ if we have a few number of cycles. Otherwise, we will use the uniform measure on the set of involution  since it has approximately 
$\frac{n}{2}$ cycles with high probability.  In this section, we denote by $\s^2:=\{\sigma\in\s,\sigma\circ\sigma=Id_n\}$ the set of involution of $\s$. If $\sigma_n$ is distributed according to the uniform distribution on $\s^2$,  the distribution of $\lambda(\sigma_n)$ on the set of Young diagrams $\mathbb{Y}_n$ is known as the Gelfand distribution. In particular, we have the following results.
\begin{proposition}\citep[Theorem 1]{meliot:hal-01215045}
If $\sigma_n$ is distributed according to the uniform distribution on $\s^2$, then 
\begin{align*}
\lim_{n\to \infty} \mathbb{P}\left(\sup_{s\in \mathbb{R}} \left|\frac{1}{\sqrt{2n}}L_{\lambda(\sigma_n)}\left({s}{\sqrt{2n}}\right)-\Omega(s)\right|<\varepsilon\right) =1.
\end{align*}
\end{proposition}
\begin{proposition}\citep[Page 692, Proposition IX.19]{9780521898065}
If $\sigma_n$ is distributed according to the uniform distribution on $\s^2$ then
\begin{align*}
    \lim_{n\to\infty}\frac{\e(card(fix(\sigma_n)))}{\sqrt{n}}=1.
\end{align*}
\end{proposition}
We will now prove the following.
\begin{corollary} \label{cor24}
If $\sigma_n$ is invariant under conjugation and supported on $\s^2$ then 
\begin{align*}
    \liminf_{n\to\infty} \frac{\e(\ell(\sigma_n))}{\sqrt{n}}\geq 2.
\end{align*}
\end{corollary}
\begin{proof}[Idea of the proof]
If $\frac{\e(card(fix(\sigma_n)))}{\sqrt{n}}\geq 2$ the result is trivial. Otherwise, the technique of  proof is identical to that of Corollary~\ref{col14}. Going back to Lemma~\ref{basiclem}, we replace $A_\sigma$ by 
$$ A'_\sigma:=\{\rho \in \s; \sigma=\rho\circ(i_1,i_2)\circ\dots\circ(i_{card(fix(\sigma))-1},i_{card(fix(\sigma))}), fix(\rho)=\emptyset\}$$
if $n$ is even and by
$$ A'_\sigma:=\{\rho \in \s; \sigma=\rho\circ(i_1,i_2)\circ\dots\circ(i_{card(fix(\sigma))-2},i_{card(fix(\sigma))-1}), card(fix(\rho))=1\}$$
if $n$ is odd. We denote by
 $T'$ the Markov operator  on $\s^2$ associated to the stochastic matrix 
$\left[\frac{\mathrm{1}_{A'_{\sigma_1}}(\sigma_2)}{card(A_{\sigma_1})}\right]_{\sigma_1,\sigma_2 \in \s^2}$.
That means that we  merge  couples of fixed points to obtain the uniform distribution on permutations having only cycles of length $2$ when $n$ is even and having an additional fixed point when $n$ is odd.
Similarly to that we did in Lemma~\ref{basiclem}, for any permutation $\sigma$, we have the following. 
\begin{itemize}
    \item [-]Almost surely, 
    \begin{equation*}
    |\ell(T'(\sigma))-\ell(\sigma)|\leq card(fix(\sigma)).
        \end{equation*}
    \item[-] More generally, almost surely,
\begin{align*} 
\max_{i\geq 1} \left|\sum_{k=1}^i\left( \lambda_k(\sigma)-\lambda_k\left(T'(\sigma)\right)\right)\right|\leq card(fix(\sigma)).
\end{align*}
\end{itemize}
Moreover, if $\sigma_n$ is invariant under conjugation, the law of $T'(\sigma_n)$ does not depend on the law of $\sigma_n$.
\\ 
 Consequently, Corollary~\ref{cor24} follows using the same techniques as in the proof of Corollary~\ref{col14}.
\end{proof}

\begin{corollary}\label{fincor}
Let  $\{\sigma_n\}_{n\geq 1}$ be a sequence of random permutations each one being  invariant under conjugation. Assume that there exists a sequence  $(\beta_n)_{n\geq 1}$ such that $$\lim_{n\to \infty}  \beta_n=+ \infty, $$ and for any $n\geq 1$, 
\begin{equation*}
    \p(card(fix(\sigma_{n}^2))>\beta_n)=1.
\end{equation*}
Then
\begin{align*}
\liminf_{n\to\infty} \frac{\e\left(\ell(\sigma_n)\right)}{\sqrt{\beta_n}}\geq 2.
\end{align*}
\begin{proof}
Giving  $A\subset \mathbb{N}$ finite, we denote  by $\mathfrak{S}_A$ (resp. $\mathfrak{S}_A^2$)  the set of permutations  (respect involutions)  of $A$. A random permutation $\sigma_A$ supported on $\mathfrak{S}_A$ is called \textit{invariant under conjugation} if for any $\sigma \in\mathfrak{S}_A$ , $\sigma \circ {\sigma}_A \circ \sigma^{-1}$  is equal in distribution to $\sigma_A$.
\\
Fix $\varepsilon>0$. By Corollary~\ref{cor24}, there exists $n_0$ such that for any  $A\subset \mathbb{N}$ with $n_0<card(A)<+\infty$, for any random permutation $\hat\sigma_A$ supported on $\mathfrak{S}_A^2$ invariant under conjugation,  
\begin{equation*}
    \frac{\e{(\ell(\hat\sigma_A)})}{\sqrt{card(A)}} \geq 2-\varepsilon.
\end{equation*}
Let $\sigma_n$ be a random permutation invariant under conjugation and $\rho_n$  be the restriction of $\sigma_n$ on $fix(\sigma_n^2)$. In particular, almost surely $\ell(\rho_n)\leq \ell(\sigma_n)$. One can see that for any $A \subset \{1,2,\dots,n\}$ such that $\p(fix(\sigma^2_n)=A)>0$, for any $\hat{\sigma}_1,\hat{\sigma}_2 \in \mathfrak{S}_A,$
\begin{equation}
    \p(\rho_n=\hat{\sigma}_1|fix(\sigma^2_n)=A)=\p(\rho_n=\hat{\sigma}_2\circ\hat{\sigma}_1\circ\hat{\sigma}_2^{-1}|fix(\sigma^2_n)=A).
\end{equation}
Consequently, if $\beta_n>n_0$,
\begin{align*}
\frac{\e\left(\ell(\sigma_n)\right)}{\sqrt{\beta_n}}& =\sum_{{}^{\quad \; \; |A|>\beta_n}_{\p(fix(\sigma^2_n)=A)>0}} \frac{\e\left(\ell(\sigma_n)|fix(\sigma_n^2)=A\right)}{\sqrt{\beta_n}} \p(fix(\sigma^2_n)=A)
\\&\geq \sum_{{}^{\quad \; \; |A|>\beta_n}_{\p(fix(\sigma^2_n)=A)>0}} (2-\varepsilon) \sqrt{\frac{card(A)}{\beta_n}}\p(fix(\sigma^2_n)=A)
\\&\geq \sum_{{}^{\quad \; \; |A|>\beta_n}_{\p(fix(\sigma^2_n)=A)>0}} (2-\varepsilon) \p(fix(\sigma^2_n)=A)=2-\varepsilon.
\end{align*}
This yields Corollary~\ref{fincor}.
\end{proof}
\end{corollary}
\begin{lemma} \label{finlem}
 For any permutation $\sigma\in \s$,
 \begin{equation*}
     card(fix(\sigma^2))\geq 6\#(\sigma)-3 card(fix(\sigma))-2n.
 \end{equation*}
\begin{proof}
We denote by $\#_k(\sigma)$ the number of cycles of  $\sigma$ of length $k$. We have
\begin{align*}
    \sum_{k\geq 1} k \#_k(\sigma)=n \quad \text{and} \quad  \sum_{k\geq 1}  \#_k(\sigma)=\#(\sigma).
\end{align*}
Thus
\begin{align*}
    n+2card(fix(\sigma))+  \#_2(\sigma)&= 3card(fix(\sigma))+ 3 \#_2(\sigma) +\sum_{k\geq 3} k \#_k(\sigma)
    \\& \geq 3  \sum_{k\geq 1}  \#_k(\sigma) = 3 \#(\sigma).
\end{align*}
Consequently,
\begin{align*}
      \#_2(\sigma)\geq  3 \#(\sigma) -n-2card(fix(\sigma)).
\end{align*}
Finally, 
 \begin{equation*}
     card(fix(\sigma^2)) = card(fix(\sigma))+ 2 \#_2(\sigma)  \geq 6\#(\sigma)-3card(fix(\sigma))-2n.
 \end{equation*}
\end{proof}
\end{lemma}
We will now prove Theorem~\ref{thm20}.
\begin{proof}
In this proof, we use the following convention. Let $A,B \subset \s$ and $f:\s\to\mathbb{R}$. If $\p(\sigma_n\in A)=0$, we assign $ \p(\sigma_n \in B|  \sigma_n \in A )=0$ and $\e(f(\sigma_n)|\sigma_n \in A)=0$.
\\ We have 
\begin{align*}
\e(\ell(\sigma_n))&=\e\left(\ell(\sigma_n)\middle|\#(\sigma_n)\leq \frac{(2+\theta)n}{6}\right)\p\left(\#(\sigma_n)\leq \frac{(2+\theta)n}{6}\right)\\&+
\e\left(\ell(\sigma_n)\middle|\#(\sigma_n)> \frac{(2+\theta)n}{6}\right)\p\left(\#(\sigma_n)> \frac{(2+\theta)n}{6}\right).
\end{align*}
Since the condition on the number of cycles is invariant under conjugation,  it is sufficient to prove Theorem \ref{thm20} in the two particular cases. 
\begin{itemize}
    \item[-] Assume that almost surely $\#(\sigma_n)\leq \frac{(2+\theta)n}{6}$.
     By Lemma \ref{lem21},  for any $0<\gamma<2$, 
    \begin{align*}
        \p\left(\frac{\ell(\sigma_n)}{\sqrt{n}}>\gamma\right)&\geq \p\left(\frac{\sum_{i=1}^{n}(\lambda_i(T(\sigma_n))-\gamma\sqrt{n})_{+}}{n}> \frac{2+\theta}{6}\right).
    \end{align*}
As $T(\sigma_n)$ is distributed according to the $Ew(0)$, by choosing $\gamma=2\sqrt{\theta}-\varepsilon$ for some $\varepsilon>0$ in Corollary \ref{col11}, we can conclude that the right hand side goes to $1$ as $n$ goes to infinity.
    \item[-] Assume that  almost surely $\#(\sigma_n)> \frac{(2+\theta)n}{6}$.
    We can write, 
    \begin{align*}
        \e(\ell(\sigma_n))&=\e(\ell(\sigma_n)|card(fix(\sigma_n)) \geq 2 \sqrt{n\theta})\p(card(fix(\sigma_n)) \geq 2 \sqrt{n\theta})\\&+
        \e(\ell(\sigma_n)|card(fix(\sigma_n)) < 2 \sqrt{n\theta})\p(card(fix(\sigma_n)) < 2 \sqrt{n\theta}).
    \end{align*}
Clearly,  if $\p(card(fix(\sigma_n)) \geq 2 \sqrt{n\theta})>0$, then  $$\e(\ell(\sigma_n)|card(fix(\sigma_n)) \geq 2 \sqrt{n\theta})>2\sqrt{n\theta}.$$ Moreover, under the condition $card(fix(\sigma_n)) < 2 \sqrt{n\theta},$ we have by Lemma~\ref{finlem},  almost surely,
\begin{equation*}
card(fix(\sigma^2_n)))>\theta n - 6\sqrt{\theta n}.    
\end{equation*}
 We can then conclude by Corollary~\ref{fincor} that if $ \p(card(fix(\sigma_n)) < 2 \sqrt{n\theta})>0$, then
 \begin{align*}
 \liminf_{n\to\infty}\frac{\e\left(\ell(\sigma_n)\middle|card(fix(\sigma_n)) < 2 \sqrt{\theta n}\right)} {\sqrt{n\theta-6\sqrt{n\theta}}}\geq2. \end{align*}
Thus, if $ \p(card(fix(\sigma_n)) < 2 \sqrt{n\theta})>0$, then \begin{align*}
 \liminf_{n\to\infty}\frac{\e\left(\ell(\sigma_n)\middle|card(fix(\sigma_n)) < 2 \sqrt{n\theta}\right)} {\sqrt{n}}
 \geq2\sqrt{\theta}. \end{align*}
    \end{itemize}
    \end{proof}

\subsection{Proof of Theorem \ref{thm8} and  Proposition~\ref{thm5}.} \label{pthm8}
The proofs of Theorem \ref{thm8} and  Proposition~\ref{thm5} are based on the following observation.
\begin{lemma}\label{lem16} For any  permutations $\sigma_{1},\sigma_{2}$, almost surely,
\begin{align*}
    |LCS(\sigma_{1},\sigma_{2})- LCS(T(\sigma_{1}),\sigma_{2})| \leq\#(\sigma_{1}).
\end{align*}
\end{lemma}
The proof is identical to that of Lemma~\ref{basiclem}.
\begin{corollary}\label{col17} Assume that the law of ${\tilde\sigma}_{1,n}$ is Ew(0) and ${\tilde\sigma}_{1,n}$ and $\sigma_{2,n}$  are independent. Then
\begin{equation*} 
\lim_{n \to \infty} \mathbb{P}\left(\frac{LCS(\tilde{\sigma}_{1,n},\sigma_{2,n})-2\sqrt{n}}{\sqrt[6]{n}}\leq s\right)=F_2(s),
\end{equation*}
\begin{align*}
\lim_{n \to \infty} \frac{\e\left({LCS(\tilde{\sigma}_{1,n},\sigma_{2,n})}\right)}{\sqrt{n}}=2 \quad \text{and} \quad \frac{{LCS(\tilde{\sigma}_{1,n},\sigma_{2,n})}}{\sqrt{n}}\overset{\mathbb{P}}{\to}2.
\end{align*}\end{corollary}
\begin{proof}
Note that if $\sigma_{1,n}$ is distributed according the uniform distribution, one can see that the independence between $\sigma_{1,n}$ and $\sigma_{2,n}$ implies that $\sigma_{1,n}^{-1}\circ\sigma_{2,n}$ follows also the uniform distribution. In this case,
\begin{equation} \label{261}
\lim_{n \to \infty} \mathbb{P}\left(\frac{LCS({\sigma}_{1,n},\sigma_{2,n})-2\sqrt{n}}{\sqrt[6]{n}}\leq s\right)=\lim_{n \to \infty}
\mathbb{P}\left(\frac{\ell(\sigma_{1,n})-2\sqrt{n}}{\sqrt[6]{n}}\leq s\right)=
F_2(s), 
\end{equation}
\begin{equation}\label{262}
    \lim_{n \to \infty} \frac{\e\left({LCS({\sigma}_{1,n},\sigma_{2,n})}\right)}{\sqrt{n}}=    \lim_{n \to \infty} \frac{\e\left({\ell({\sigma}_{1,n})}\right)}{\sqrt{n}}=2,
\end{equation}
and
\begin{equation}\label{263}
    \frac{{LCS({\sigma}_{1,n},\sigma_{2,n})}}{\sqrt{n}}\overset{d}{=}\frac{\ell(\sigma_{1,n})}{\sqrt{n}}\overset{\mathbb{P}}{\to}2.
\end{equation}
The second equality of \eqref{261} is due to \cite*{Baik} and the second equality  of \eqref{262} and  the convergence of \eqref{263} are due to \cite{MR0480398}. Hence, one can conclude by Lemma~\ref{lem16} since ${\e(\#(\sigma_{1,n}))=log(n)+O(1)}$ and $LCS(\tilde{\sigma}_{1,n},\sigma_{2,n})$ is equal in distribution to $LCS(T({\sigma}_{1,n}),\sigma_{2,n})$.
\end{proof}
Using again Lemma~\ref{lem16}, Corollary \ref{col17} imply Proposition~\ref{thm5} since $T({\sigma}_{1,n})$ is distributed according to $Ew(0)$. 

\begin{proof}[Sketch of proof of Theorem~\ref{thm8}]
Using the same technique as in Corollary \ref{col11}, we can prove that for any $\varepsilon>0$, 
\begin{align*}
    \p\left(\frac{LCS(\sigma_{1,n},\sigma_{2,n})}{\sqrt{n}}  > G^{-1}\left(\frac{\#(\sigma_{1,n})}{2n}+\varepsilon\right)-\varepsilon\right)\to 1.
\end{align*} Consequently,\begin{align*}
    \liminf_{n\to \infty} \frac{\e(LCS(\sigma_{1,n},\sigma_{2,n}))}{\sqrt{n}}\geq \e\left(G^{-1} \left({\liminf_{n\to\infty}} \frac{\#(\sigma_{1,n})}{2n}\right)\right). 
\end{align*}
Since $G^{-1}$ is convex, we can conclude using Jensen's inequality.
\end{proof}

\end{document}